\documentclass[11pt,a4paper]{article}
\usepackage{amssymb}
\usepackage{amsmath}
\usepackage{float}
\usepackage{subcaption}
\usepackage{tikz}

\font\Bbb=msbm10 at 10 truept
\def\n{\hbox{\Bbb N}}
\def\z{\hbox{\Bbb Z}}
\DeclareMathOperator{\Reg}{Reg}
\DeclareMathOperator{\Base}{Base}
\DeclareMathOperator{\Ann}{Ann}
\DeclareMathOperator{\Ima}{Im}

\newcommand{\qed}{
  \ifmmode
   \eqno{\qedsymbol}
  \else
    \leavevmode\unskip\penalty9999 \hbox{}\nobreak\hfill\hbox{\qedsymbol}
  \fi
}
\newcommand{\qedsymbol}{\leavevmode\vrule height 1.2ex width 1.1ex depth -.1ex}
\newenvironment{proof}{\begin{trivlist}\item[\hskip
\labelsep{\bf Proof.\quad}]}
{\hfill\qed\rm\end{trivlist}}

\newcommand{\xux}[1]{#1U_{\overline{#1}}}
\newcommand{\ed}{\varepsilon\delta}
\newcommand{\pn}{{\cal P}(N)}

\newtheorem{theorem}{Theorem}[section]
\newtheorem{corollary}[theorem]{Corollary}
\newtheorem{proposition}[theorem]{Proposition}
\newtheorem{lemma}[theorem]{Lemma}

\title{The multiplicative semigroup of a Dedekind domain}
\author{James Renshaw\\
\small School of Mathematical Sciences\\
\small University of Southampton\\
\small Southampton, SO17 1BJ, England\\
\small ORCID: 0000-0002-5571-8007\\
\small  j.h.renshaw@soton.ac.uk\\\\
William Warhurst\\
\small School of Mathematical Sciences\\
\small University of Southampton\\
\small Southampton, SO17 1BJ, England\\
\small  w.warhurst@soton.ac.uk\\
}

\begin{document}
\date{August 2023}
\maketitle
\begin{abstract}
\noindent In 1995 Grillet defined the concept of a stratified semigroup and a stratified semigroup with zero. The present authors extended that idea to include semigroups with a more general base and proved, amongst other things, that finite semigroups in which the ${\cal H}-$classes contain idempotents, are semilattices of stratified extensions of  completely simple semigroups, and every strict stratified extension of a Clifford semigroup is a semilattice of stratified extensions of groups. We continue this work here by considering the multiplicative semigroup of Dedekind domains and show in particular that quotients of such rings have a multiplicative structure that is a (finite) boolean algebra of stratified extensions of groups.
\end{abstract}
\section{Introduction and preliminaries}

In \cite{warhurst-23}, the authors introduce stratified extensions of semigroups as a generalisation of the work of Grillet in \cite{grillet-95}. They define the \textit{base} of a semigroup $S$ to be the subset $\Base(S) = \bigcap_{m>0} S^m$ and note that $S$ is a \textit{stratified semigroup} as defined by Grillet if $\Base(S) = \{0\}$ or $\Base(S)$ is empty. A semigroup $S$ is then called a \textit{stratified extension} of $\Base(S)$ if $\Base(S) \neq \emptyset$. The name signifying the fact that in this case $S$ is an ideal extension of $\Base(S)$ by a stratified semigroup with zero.

\medskip

After a few preliminaries, in Section 2 we consider the multiplicative structure of commutative rings in a more general way and describe the ${\cal J}-$classes in terms of certain annihilators. We then show that the multiplicative semigroup can be viewed as a semilattice of semigroups. In section 3 we specialise to Dedekind domains and show that the subsemigroups of the semilattice are stratified extensions of groups. In section 4, we consider quotients of Dedekind domains and demonstrate by using prime factorisations of ideals, that the multiplicative structure is a finite Boolean algebra of stratified extensions of groups and give a `recipe' for constructing both the semilattice and the stratified subsemigroups. Section 5 then presents some interesting examples.

\medskip

If $S$ is a stratified extension of $\Base(S)$, the \textit{layers} of $S$ are defined to be the sets $S_m = S^m \setminus S^{m+1}$, $m\ge 1$. This definition makes sense for any semigroup, of course, but we are only interested in the case when $\Base(S) = \cap_{m\ge 1}S^m\ne\emptyset$. Every element of $S$ lies either in the base of $S$ or in exactly one layer of $S$, and if $s \in S_m$ then $m$ is the \textit{depth} of $s$. If $S$ has finitely many layers then the numbers of layers is called the \textit{height} of $S$. The layer $S_1$ generates every element of $S \setminus \Base(S)$ and is contained in any generating set of $S$.

\medskip

Since $\Base(S) \subseteq S^m$ for any $m \in \mathbb{N}$, we have an alternative characterisation for the elements of $\Base(S)$. A element $s \in S$ lies in $\Base(S)$ if and only if for any $m \in \mathbb{N}$, $s$ can be factored into a product of $m$ elements i.e. $s = a_1 a_2 \dots a_m$ for some $a_i \in S$. This characterisation allows us to deduce some immediate properties of $\Base(S)$ as a subsemigroup of $S$.

\begin{lemma}[{\cite[Corollary 2.2]{warhurst-23}}]\label{paper1-lemma}
Let $S$ be a semigroup. $\Reg(S)\subseteq \Base(S)$ and if $S$ is regular then $\Base(S) = S$.
\end{lemma}

\begin{lemma}\label{paper1-lemma2}
Let $S,T$ be semigroups and $f:S\to T$ a morphism. Then for all $i\in\n$, $S^i\subseteq f^{-1}(T^i)$ and so in particular, $\Base(S)\subseteq f^{-1}(\Base(T))$.
\end{lemma}

\medskip

For the basic concepts in semigroup theory we refer the reader to~\cite{howie-95}. In particular, we say that $S$ is a {\em semilattice of semigroups} $S_\alpha, \alpha\in Y$, and write $S = {\cal S}[Y,S_\alpha]$, if $Y$ is a semilattice, $S = \dot\cup_{\alpha\in Y}S_\alpha$ and for all $\alpha,\beta \in Y$, $S_\alpha S_\beta\subseteq S_{\alpha\beta}$. Notice that if $Y$ is a semilattice and if $\phi : S\to Y$ is an onto morphism, then $S$ is a semilattice of semigroups $S = {\cal S}[Y,\phi^{-1}(\alpha)]$.

\medskip

For more details of basic definitions and results in ring theory, we refer the reader to~\cite{cohn-82} and~\cite{cohn-89}. An ideal $I$ of a ring $R$ is {\em prime} if $I$ is a proper ideal and for all $a,b\in R$, if $ab\in I$ then either $a\in I$ or $b\in I$. A {\em domain} is a ring with no non-zero zero-divisors and a {\em Dedekind domain} is a commutative domain in which every non-zero proper ideal can be factored into a product of prime ideals. If $I,J\unlhd R$ are ideals of $R$ then we say that {\em $I$ divides $J$} and write $I|J$ if and only if there exists $H\unlhd R$ with $I = JH$. Then $R$ is a Dedekind domain if and only if
$$
\text{for all }I,J\unlhd R, J\subseteq I\text{ if and only if }I|J.
$$
A {\em principal ideal domain} is a commutative domain in which every ideal is principal. A Dedekind domain is a principal ideal domain if and only if it is a unique factorisation domain. If $R$ is a Dedekind domain and $\{0\}\ne I\unlhd R$ is a non-zero ideal of $R$ then $R/I$ is a principal ideal ring. A Dedekind domain is Noetherian and as such every non-zero, non-unit element can be factorised into a product of irreducible elements. The following elementary properties of ideals will be used implicitly in some of what follows.

\begin{lemma}
Let $I,J$ be ideals of $R$. Then
\begin{enumerate}
\item $IJ \subseteq I \cap J$.
\item $I \cup J \subseteq I + J$
\item $I \subseteq J \iff I + J = J$.
\item $IJ+J = J$.
\end{enumerate}
\end{lemma}

\section{Rings as semilattices of semigroups}\label{semilattice-section}
Let $R$ be a commutative ring with unity. We will show that the multiplicative semigroup of $R$ is a semilattice of semigroups, and investigate this structure further in Section~\ref{section-3}.
Note that as $R$ is commutative then on the multiplicative semigroup of $R$, Green's relations, ${\cal H} = {\cal R}={\cal L} = {\cal D} = {\cal J}$ coincide. Recall that for any ring $R$, the quotient $R/(0)$ is naturally isomorphic to $R$ itself via the map $x \mapsto x + (0)$. Some of the results below take place within $R/(0)$ and we could make use of this isomorphism to recast them in $R$ instead. However we have chosen not to do this explicitly.

Let $D$ be the set of all ideals of $R$. It is easy to see that, under the usual addition and multiplication of ideals, $D$ forms a semiring with additive identity $(0)$ and multiplicative identity $(1) = R$. Let $\delta: R \rightarrow D$ be given by $\delta(x) = (x)$. This is clearly a semiring homomorphism.

\begin{proposition}
The kernel of $\delta$ is Green's ${\cal J}-$relation and hence ${\cal J}$ is a congruence on $R$.
\end{proposition}

\begin{proof}
Let $x, y \in R$ such that $\delta(x) = \delta(y)$. Then the principal ideals $RxR$ and $RyR$ are equal, so $x {\cal J} y$. Conversely if $x {\cal J} y$ then $\delta(x) = RxR = RyR = \delta(y)$.
\end{proof}

Let $x\in R$ and let
$$
\overline{x} = \Ann(x) = \{y \in R \;|\; xy = 0\}
$$
be the annihilator of $x$, which is clearly an ideal of $R$. Let $R_{\overline{x}} = R / \overline{x}$ and let $U_{\overline{x}}$ be the group of units of this quotient. For $y \in R$, we denote by $[y]_{\overline{x}}$ the coset $y + \overline{x}$ and consider the set $\xux{x}$. We will see that this set is essentially the ${\cal J}-$class of $R$ containing $x$. Note that $x\overline{x} = \{xy \;|\; y \in R, xy = 0\} = \{0\} = (0)$.

\begin{lemma}\label{cancellation-lemma}
Let $x\in R$ and let $V_x = \{u\in R\;|\;\exists v\in R, xuv=x\}$. Then
\begin{enumerate}
\item $\xux{x}\subseteq xR_{\overline x}\subseteq R/(0)$,
\item For $[u]_{\overline x},[v]_{\overline x} \in R_{\overline x}$,  $x[u]_{\overline{x}} = x[v]_{\overline{x}}$ if and only if $[u]_{\overline{x}} = [v]_{\overline{x}}$,
\item $V_x$ is a submonoid of $R$ and $u\in V_x$ if and only if $[u]_{\overline x} \in U_{\overline x}$.
\end{enumerate}
\end{lemma}

\begin{proof}
\begin{enumerate}
\item For any $u \in R$ we have $x[u]_{\overline{x}} = x(u+\overline{x}) = xu + x\overline{x} = xu + (0) \in R/(0)$ and so  $xR_{\overline{x}} \subseteq R/(0)$.

\item Let $x[u]_{\overline{x}} = x[v]_{\overline{x}}$. Then $xu-xv\in(0)$ and so $x(u-v)=0$. Hence $u-v \in \overline{x}$ and so $[u]_{\overline{x}} = [v]_{\overline{x}}$. The converse is obvious.

\item That $V_x$ is a submonoid of $R$ is fairly clear. Suppose that $[u]_{\overline x}\in U_{\overline x}$ so that there exists $[v]_{\overline x} \in U_{\overline x}$ such that $[u]_{\overline x}[v]_{\overline x}=[1]_{\overline x}$. Then $uv-1\in{\overline x}$ and so $xuv=x$. Conversely, if $u,v\in R$ such that $xuv=x$ then $x[u]_{\overline x}[v]_{\overline x} = xuv+(0) = x+(0) = x[1]_{\overline x}$ and so from part (2) it follows that $[u]_{\overline x}\in U_{\overline x}$.
\end{enumerate}
\end{proof}

\begin{theorem} \label{xux-theorem}
Let $x,y \in R$. Then $x {\cal J} y$ if and only if $y+(0) \in \xux{x}$. Consequently the sets $\xux{x}$ are the ${\cal J}-$classes of $R/(0)$.
\end{theorem}

\begin{proof}
Suppose $y+(0) \in \xux{x}$. Then $y+(0) = x[u]_{\overline{x}} = xu+(0)$ for some $u \in V_x$, and so $y=xu$. Since $[u]_{\overline{x}}$ is a unit, there exists $v \in V_x$ such that $xuv=x$ and so $yv+(0) = xuv+(0) = x+(0)$ and hence $x=yv$ and so $x {\cal J} y$.

\smallskip

Now let $x {\cal J} y$. Then there exists $u \in R$ such that $xu = y$. Then $y+(0) = xu+(0) = x[u]_{\overline{x}}$. Further, there exists $v \in R$ such that $x = yv$. Then $x = xuv$ and so from Lemma~\ref{cancellation-lemma}(3), $[u]_{\overline x}\in U_{\overline x}$ and $y+(0) \in \xux{x}$.
\end{proof}

From this and the fact that ${\cal J} = \ker(\delta)$ we immediately deduce
\begin{corollary}\label{xux-corollary}
Let $x,y\in R$. Then
$$
x {\cal J} y\text{ if and only if }\xux{x} = \xux{y}\text{ if and only if }(x) = (y).
$$
\end{corollary}

Notice that $x{\cal J}_R y$ if and only if $\left(x+(0)\right){\cal J}_{R/(0)}\left(y+(0)\right)$ and that the sets $\xux{x}$ are the ${\cal J}-$classes of $R/(0)$. Consequently, if $x_1 +(0), x_2+(0) \in \xux{x}$ then
$$
x_1\;{\cal J}_{R}\;x_2\;{\cal J}_{R}\;x
$$
and so since ${\cal J}_R=\ker{\delta}$
$$
(x_1)=(x_2) = (x).
$$

Note also that since $u \in V_x$ if and only if $[u]_{\overline{x}} \in U_{\overline{x}}$ then $y \in xV_x$ if and only if $y+(0) \in \xux{x}$. It follows that $xV_x$ is the image of $\xux{x}$ under the natural isomorphism from $R/(0)$ to $R$ and hence is the ${\cal J}$-class of $R$ containing $x$. Our reason for working with $\xux{x}$ rather than $xV_x$ is due to Lemma~\ref{cancellation-lemma}(2): if $u,v \in V_x$ and $xu = xv$ then it is not necessarily the case that $u = v$. For example, if $e \in R$ is a non-unit idempotent then $1,e \in V_e$ and $e1 = ee = e$ but $e \neq 1$. This cancellation property is required for the following theorem.

\begin{theorem} \label{xux-group-theorem}
Let $R$ be a ring and let $x\in R$. The set $\xux{x}$ is a subsemigroup of $R/(0)$ if and only if $\delta(x)$ is an idempotent. It is in fact a subgroup which is  isomorphic to $U_{\overline x}$.
\end{theorem}

\begin{proof}
Let $x_1+(0), x_2+(0) \in \xux{x}$. Then from above, $x_1{\cal J}x_2{\cal J}x$ and so since ${\cal J}$ is a congruence, it follows that $x_1x_2{\cal J}x^2$ and hence from Corollary~\ref{xux-corollary}, $x_1x_2+(0) \in \xux{x^2}$. But if $\xux{x}$ is a subsemigroup of $R/(0)$ then $x_1x_2+(0)\in \xux{x}$ and so $\xux{x}\cap\xux{x^2}\ne\emptyset$. Consequently $\xux{x} = \xux{x^2}$ and so $\delta(x) = (x)=(x^2)=\delta(x)^2$ by Corollary~\ref{xux-corollary}.

Conversely, if $(x) = (x^2)$ then $x{\cal J}x^2$ and so in particular $x+(0)$ and $x^2+(0)$ belong to the same ${\cal H}$-class of $R/(0)$, $\xux{x}$, and hence  this ${\cal H}$-class is a group.

\medskip

If $(x) = (x^2)$ then $x = x^2k$ for some $k\in R$ and so $x \in V_x$. Now define $\phi:\xux{x}\to U_{\overline x}$ by $\phi(x[u]_{\overline x}) = [xu]_{\overline x}$. By Lemma~\ref{cancellation-lemma}(3) this map is well-defined and it is clearly onto. In addition
\begin{gather*}
\phi(x[u]_{\overline x}x[v]_{\overline x}) = \phi(x^2[uv]_{\overline x}) = \phi(x[xuv]_{\overline x})\\ = [x^2uv]_{\overline x} = \left([xu]_{\overline x}\right)\left([xv]_{\overline x}\right)\\ = \phi(x[u]_{\overline x})\phi(x[v]_{\overline x}),
\end{gather*}
and so $\phi$ is a morphism. Finally, if $\phi(x[u]_{\overline x}) = \phi(x[v]_{\overline x})$ then $[xu]_{\overline x} = [xv]_{\overline x}$ and so $[x]_{\overline x}[u]_{\overline x} = [x]_{\overline x}[v]_{\overline x}$. Hence $[u]_{\overline x} = [v]_{\overline x}$ since $[x]_{\overline x}\in U_{\overline x}$. Therefore $x[u]_{\overline x} = x[v]_{\overline x}$ as required.
\end{proof}

\bigskip

For any $I \in D$, let $E_I = \{J \in E(D) \;|\; J \subseteq I\}$. This set is non-empty since $(0) \in E(D)$ and $(0) \subseteq I$ for any $I \in D$. We claim that
$$
\varepsilon(I) = \sum_{J \in E_I} J
$$
is the greatest element of $E_I$ with respect to subset inclusion. It is easy to see that $\varepsilon(I) \in D$ and for every $J \in E_I$, $J \subseteq \varepsilon(I)$. It remains to show that $\varepsilon(I) \in E(D)$ and $\varepsilon(I) \subseteq I$.

\smallskip

A general element of $\varepsilon(I)$ has the form $e_1 + \ldots + e_n$ where each $e_i$ lies in some ideal $J_i \in E_I$. Hence every element of $\varepsilon(I)$ is an element of $I$ and so $\varepsilon(I) \subseteq I$.
As $J_i \in E(D)$, $e_i \in J_iJ_i$. Then
$$
e_1 + \dots +e_n \in J_1J_1 + \ldots + J_nJ_n \subseteq (J_1 + \ldots + J_n)(J_1 + \ldots + J_n).
$$
Since each $J_i \in E_I$, $J_1 + \ldots + J_n \subseteq \varepsilon(I)$ so $\varepsilon(I) \subseteq \varepsilon(I)\varepsilon(I)$. The reverse inclusion holds for any ideal and so $\varepsilon(I) \in E(D)$ as required.

\medskip

This construction clearly describes a well-defined map $\varepsilon:D\to E(D)$. For each $e \in E(D)$ let
$$
D_e = \varepsilon^{-1}(e).
$$

\begin{proposition}\label{D-semilattice-proposition}
The multiplicative semigroup of $D$ is a semilattice of semigroups ${\cal S}[E(D);D_e]$.
\end{proposition}

\begin{proof}
As $R$ is commutative, $D$ is also commutative and hence $E(D)$ is a semilattice. If $I \in E(D)$ then clearly $I$ is the greatest element of $E_I$ so $\varepsilon(I) = I$ and $\varepsilon$ is a surjection onto the semilattice $E(D)$. It remains to show that $\varepsilon$ is a homomorphism.

Let $I, J \in D$ and $K \in E(D)$. If $K \in E_I \cap E_J$ then $K \subseteq I$ and $K \subseteq J$. Then $K = KK \subseteq IJ$ so $K \in E_{IJ}$. Conversely, if $K \in E_{IJ}$ then $K \subseteq IJ$. But $IJ \subseteq I$ so $K \subseteq I$ and $K \in E_I$. In a similar way, $K \in E_J$ and hence $K \in E_I \cap E_J$ and $E_{IJ} = E_I \cap E_J$.

It is easily seen that for all $L\in D$, $E_L = E_{\varepsilon(L)}$ and hence $$E_{\varepsilon(IJ)} = E_{IJ} = E_I \cap E_J = E_{\varepsilon(I)} \cap E_{\varepsilon(J)} = E_{\varepsilon(I)\varepsilon(J)}.$$ As $E(D)$ is a semilattice $\varepsilon(I)\varepsilon(J) \in E(D)$ and so it follows that $\varepsilon(IJ) = \varepsilon(I)\varepsilon(J)$ as required.
\end{proof}

For each $e \in E(D)$ let $R_e = (\ed)^{-1}(e)$.

\begin{theorem} \label{R-semilattice-theorem}
The multiplicative semigroup of $R$ is a semilattice of semigroups ${\cal S}[\Ima(\ed);R_e]$.
\end{theorem}

\begin{proof}
It is clear that $\ed$ is a surjective homomorphism from $R$ onto its image. Since $\Ima(\ed)$ is a subsemigroup of the semilattice $E(D)$, it is also a semilattice.
\end{proof}

We now want to consider the nature of the semigroups $D_e$ and $R_e$ for a specific type of ring.

\section{Dedekind domains}\label{section-3}

Let $R$ be a Dedekind domain. We wish in the next section to consider quotients of Dedekind domains but we first make some observations about Dedekind domains in general. While the semigroup structure of these rings is not too complex, it is interesting in its own right. We will show that $R$ is a semilattice of stratified extensions of groups. More specifically, $R$ is a semilattice of two semigroups; its group of units and a stratified extension of the trivial group. Recall that $D$ is the collection of all ideals of $R$ and that $\varepsilon\delta(x)$ is the largest idempotent ideal contained in $(x)$.

\begin{proposition}
The idempotents of $D$ are $R$ and $(0)$.
\end{proposition}

\begin{proof}
As $R$ is a Dedekind domain, every non-zero proper ideal $I$ of $R$ can be factorised uniquely as a product of prime ideals, so $I = X_1\ldots X_n$ for some prime ideals $X_i\unlhd R$. If $I \in E(D)$ then $I = I^2 = X_1\ldots X_nX_1\ldots X_n$ is another factorisation of $I$ into prime ideals. This contradicts the uniqueness of the factorisation, and so $I$ cannot be idempotent. Hence there are no idempotent non-zero proper ideals of $R$ and so $E(D) = \{R, (0)\}$.
\end{proof}

Clearly for any ideal $I$ of $R$ we have $(0) \subseteq I \subseteq R$ and so $\varepsilon(I) = R$ if $I = R$ and $\varepsilon(I) = (0)$ otherwise. Note also that $R = \ed(1)$ and $(0) = \ed(0)$ so $\ed$ is surjective, and in addition $D_R$ is the trivial semigroup and hence is vacuously a stratified extension of a group. 

\begin{proposition}
The subsemigroup $D_{(0)}$ is a stratified extension of the trivial group $\{(0)\}$.
\end{proposition}

\begin{proof}
Let $I$ be a non-zero proper ideal of $R$ so $I \in D_{(0)}$. Suppose $I$ factors uniquely as a product of $n$ prime ideals, $I = X_1\ldots X_n$. Each $X_i$ is a non-zero proper ideal of $R$ so lies in $D_{(0)}$ and hence $I \in {D_{(0)}}^n$. If $I \in {D_{(0)}}^{n+1}$ then $I = Y_1\ldots Y_{n+1}$ for some $Y_i \in D_{(0)}$. Since $I$ is non-zero, clearly each $Y_i$ is non-zero and so factors as a product of prime ideals. But then $I$ can be written as a product of at least $n+1$ prime ideals, contradicting the uniqueness of the previous factorisation. Hence $I \not\in {D_{(0)}}^{n+1}$ and so $I \in {D_{(0)}}^n \setminus {D_{(0)}}^{n+1}$. As this holds for every non-zero ideal of $R$, we have $\Base(D_{(0)}) = \{(0)\}$ and hence $D$ is a stratified extension of the trivial group.
\end{proof}

\begin{theorem}\label{domain-semilattice-theorem}
Let $R$ be a Dedekind domain. Then $R$ is a semilattice of stratified extensions of groups.
\end{theorem}

\begin{proof}
Since $\ed$ is a surjection, by Theorem~\ref{R-semilattice-theorem}, $R$ is a semilattice of semigroups ${\cal S}[E(D);R_e]$. Clearly $\delta(x) = R$ if and only if $x$ is a unit of $R$, and so $R_R$ is exactly the group of units of $R$. For $R_{(0)}$, by Lemma~\ref{paper1-lemma2}, $\Base(R_{(0)}) \subseteq \delta^{-1}(\Base(D_{(0)})) = \delta^{-1}((0)) = \{0\}$. Since $0$ is idempotent we have $0 \in \Base(R_{(0)})$ and so $\Base(R_{(0)}) = \{0\}$ and hence $R_{(0)}$ is a stratified extension of the trivial group.
\end{proof}

Note that $\Base(R_{(0)}) = \delta^{-1}(\Base(D_{(0)}))$. In general the layers within the stratified structure of $D_0$ and within $R_0$ will not be the same. However,

\begin{proposition}\label{pid-proposition}
Let $R$ be a Dedekind domain and let $i>1$. Then ${R_{(0)}}^i = \delta^{-1}({D_{(0)}}^i)$ if and only if $R$ is a principal ideal domain.
\end{proposition}

\begin{proof}
Note that by Lemma~\ref{paper1-lemma2}, ${R_{(0)}}^i\subseteq \delta^{-1}({D_{(0)}}^i)$ is always true for any commutative ring $R$.

Let $R$ be a principal ideal domain and let $x \in \delta^{-1}({D_{(0)}}^i)$. Then $\delta(x)=(x)$ can be factorised as a product of $i$ principal ideals $(x) = (x_1)\ldots(x_i) = (x_1\ldots x_i)$ with each $(x_j) \in D_{(0)}$. Hence for $1\le j\le i$, $x_j \in R_{(0)}$ and so $x_1\ldots x_i \in R_{(0)}$. Since $\delta(x) = \delta(x_1\ldots x_i)$ we have $x {\cal J} x_1\ldots x_i$ and so $x = x_1\ldots x_iu$ for some $u \in R$. Since
$$
\ed(x_iu) = \ed(x_i)\ed(u) = (0)\ed(u) = (0)
$$
then $x_iu \in R_{(0)}$ and it follows that $x \in {R_{(0)}}^i$.

For the converse, note that since $R$ is a Dedekind domain it is Noetherian and hence every non-zero, non-unit element can be factorised into a product of irreducible elements. It follows that every irreducible element of $R$ is prime if and only if $R$ is a unique factorisation domain and hence a principal ideal domain. Hence if $R$ is not a principal ideal domain there exists some $x \in R_{(0)}$ such that $x$ is irreducible but not prime (recall that $R_R$ consists of units which are not irreducible). Since $x$ is irreducible it cannot be written as a product of two non-unit elements of $R$ and hence $x \not\in {R_{(0)}}^2$. Since $x$ is not prime, $(x)$ is not a prime ideal and so has a unique factorisation as a product of prime ideals $X_1\ldots X_n$ for some $n>1$. In particular, $(x) \in {D_{(0)}}^2$ and so ${R_{(0)}}^2 \neq \delta^{-1}({D_{(0)}}^2)$. It is then an easy matter to extend this for all $i\ge2$.
\end{proof}

Notice then that the $i$-th layer of $R_{(0)}$ is the preimage of the $i$-th layer of $D_{(0)}$.

\begin{corollary}
An element $x \in R$ is prime if and only if $(x)$ lies in the first layer of $D_{(0)}$. Additionally, if $x \in R$ is prime then $x$ lies in the first layer of $R_{(0)}$. The converse holds only when $R$ is a PID.
\end{corollary}

\section{Quotients of Dedekind domains}\label{section-4}

Let $S$ be a Dedekind domain, $A\unlhd S$ and let $R = S/A$. We will demonstrate that $R$ is a semilattice of stratified extensions of groups. Note that when $A=(0), R \cong S$ and this case has effectively been considered in Section~\ref{section-3}. When $A=S$ then $R=\{0\}$ and this situation is trivial. Hence we shall assume in what follows that $S\ne A\ne(0)$.

\medskip

Let $D_A$ be the set of ideals of $S$ containing $A$ and define an operation $\ast$ on $D_A$ such that $X \ast Y = XY + A$. Then $$(X \ast Y) \ast Z = (XY + A) \ast Z = (XY + A)Z + A = XYZ + AZ + A = XYZ + A$$ and similarly $X \ast (Y \ast Z) = XYZ + A$ and so $\ast$ is associative. Note that for every $I \in D_A$, $I + A = I$.

\medskip

The following is well known (see for example~\cite[Third Isomorphism Theorem, Page 303]{cohn-82}), but as the result is normally presented as an isomorphism $D_A\to D$, we feel the proof is useful to present here.
\begin{lemma}
The map $\Phi: D \rightarrow D_A$ given by $\Phi(I) = \bigcup_{X \in I} X$ is an isomorphism.
\end{lemma}

\begin{proof}
Note that $x+A\in I$ if and only if $x \in \Phi(I)$.
\smallskip

We first show $\Phi$ is well defined. Let $x,y \in \Phi(I)$. Then $x+A, y+A \in I$ so $x+y+A \in I$ and hence $x+y \in \Phi(I)$. Similarly for any $z \in S$, $z + A \in R$ so $xz+A \in I$ and $xz \in \Phi(I)$ and hence $\Phi(I)$ is an ideal of $S$. Since $0 + A \in I$, $A \subseteq \Phi(I)$ and so $\Phi(I)\in D_A$.

To see that $\Phi$ is injective, if $\Phi(I) = \Phi(J)$ then we have $$x + A \in I \Leftrightarrow x \in \Phi(I) \Leftrightarrow x \in \Phi(J) \Leftrightarrow x + A \in J$$ so $I = J$.
For surjectivity, let $I$ be an ideal of $S$ containing $A$. Then $J = \{x + A \;|\; x \in I\}$ is clearly an ideal of $R$ and $\Phi(J) = I$.

Finally we show that $\Phi$ is a homomorphism. If $x \in \Phi(I) \ast \Phi(J) = \Phi(I)\Phi(J) + A$ then $x = x_1y_1 + \ldots  + x_ny_n +a$ where $a\in A$, $x_i +A \in I$ and $y_i +A \in J$ for each $i \in \{1,\ldots ,n\}$. Then $x_1y_1 + \ldots  + x_ny_n + A \in IJ$ so $x_1y_1 + \ldots  + x_ny_n \in \Phi(IJ)$ and $x_1y_1 + \ldots  + x_ny_n + a \in \Phi(IJ)+A = \Phi(IJ)$. Hence $\Phi(I) \ast \Phi(J) \subseteq \Phi(IJ)$. For the reverse inclusion let $x \in \Phi(IJ)$ so $x + A \in IJ$ and $x + A = (x_1 + A)(y_1 + A) + \ldots  + (x_n + A)(y_n + A) = x_1y_1 + \ldots  + x_ny_n + A$. Then $x - (x_1y_1 + \ldots  + x_ny_n) \in A$ so $x = x_1y_1 + \ldots  + x_ny_n + a$ for some $a \in A$. Hence $x \in \Phi(I)\Phi(J) + A = \Phi(I) \ast \Phi(J)$. Therefore $\Phi(I) \ast \Phi(J) = \Phi(IJ)$ and $\Phi$ is a homomorphism.
\end{proof}

Notice that if $K\in D_A$ then $\Phi^{-1}(K) = K/A$.

\medskip

Since $S$ is a Dedekind domain, every nonzero proper ideal factors into a product of prime ideals. Hence for all $I\in D, I\ne R$, $\Phi(I) = P_1P_2\ldots P_n$ for some prime ideals $P_i$ of $S$. Then $\Phi(I) = \Phi(I)+A = P_1\ldots P_n + A$ since $A \subseteq \Phi(I)$. For each $1\le i\le n$, $A \subseteq \Phi(I) \subseteq P_i$ so $P_i \in D_A$. Hence $\Phi(I) = P_1\ldots P_n + A = P_1 \ast \ldots  \ast P_n = \Phi(X_1) \ast \ldots  \ast \Phi(X_n)$ where $X_i = \Phi^{-1}(P_i) = P_i/A$ and so $I = X_1\ldots X_n$. Note that this is not necessarily a unique factorisation, as for example $(4)$ as an ideal of $\z_{12}$ can be written as $(2)(2)$ or as $(2)(2)(2)$. It is however a factorisation into prime ideals.

\begin{lemma}
The ideal $I$ is a prime ideal of $R$ if and only if $\Phi(I)$ is a prime ideal of $S$. 
\end{lemma}

\begin{proof}
Suppose $\Phi(I)$ is a prime ideal of $S$ and let $(x +A)(y + A) = xy + A \in I$. Then $xy \in \Phi(I)$ and so without loss of generality $x \in \Phi(I)$. Hence $x + A \in I$ and so $I$ is a prime ideal. Conversely, suppose $I$ is a prime ideal of $R$ and let $xy \in \Phi(I)$. Then $xy + A \in I$ so without loss of generality $x + A \in I$ and hence $x \in \Phi(I)$ so $\Phi(I)$ is prime.
\end{proof}

The following lemma shows that the factorisation of $I$ into $\Phi^{-1}(P_1)\ldots \Phi^{-1}(P_n)$ is a {\em minimal prime factorisation}, in the sense that any other prime factorisation of $I$ must include each of these factors.

\begin{lemma}\label{uniqueness-lemma}
Let $I \in D$ be such that $\Phi(I)$ has a unique prime factorisation $P_1\ldots P_n$. If $X_1\ldots X_m$ is a prime factorisation of $I$ then $m \geq n$ and, up to reordering factors, $X_i = \Phi^{-1}(P_i)$ for $i \in \{1,\ldots ,n\}$.
\end{lemma}

\begin{proof}
By definition, $$\Phi(I) = \Phi(X_1) \ast \ldots  \ast \Phi(X_m) = \Phi(X_1)\ldots \Phi(X_m) + A$$ so $\Phi(X_1)\ldots \Phi(X_m) \subseteq \Phi(I)$ and hence $\Phi(I)$ divides $\Phi(X_1)\ldots \Phi(X_m)$ as $S$ is a Dedekind domain. Then $P_1\ldots P_nQ = \Phi(X_1)\ldots \Phi(X_m)$ for some ideal $Q$ of $S$ so, by uniqueness of prime factorisations in $S$, we have $m \geq n$ and, reordering if necessary, $P_i = \Phi(X_i)$ for each $i \in \{1,\ldots ,n\}$. Applying $\Phi^{-1}$ to each equality then gives the desired result. 
\end{proof}

Suppose $A$ has prime factorisation $P_1^{e_1}\ldots P_n^{e_n}$ ($e_i > 0$). By definition, any ideal $\Phi(I) \in D_A$ has $A \subseteq \Phi(I)$ and so $\Phi(I)$ divides $A$ and hence $\Phi(I) = P_1^{f_1}\ldots P_n^{f_n}$ where $0 \leq f_i \leq e_i$. In particular, if $P$ is a prime ideal of $S$ then $A \subseteq P$ if and only if $P = P_i$ for some $i \in \{1,\ldots ,n\}$. Let
$$
A_i = \Phi^{-1}(P_i)
$$
for each $i \in \{1,\ldots ,n\}$. Then $A_1, \ldots , A_n$ are precisely the prime ideals of $R$ and any $I \in D$ has minimal prime factorisation $A_1^{f_1}\ldots A_n^{f_n}$, for some $f_i\ge0$. Notice that the primes $A_1, \ldots, A_n$ are unique with respect to this construction, by Lemma~\ref{uniqueness-lemma}.

\medskip

Note here that we adopt the convention $P_1^0,\ldots, P_n^0 = S$ and $A_1^0,\ldots, A_n^0 = R$, i.e. that the empty powers of primes are the identity elements of $D_A$ and $D$ respectively.

\begin{lemma}\label{minimal-factorisation-lemma}
Let $I \in D$. If $I$ has prime factorisation $A_1^{g_1}\ldots A_n^{g_n}$ then the minimal prime factorisation of $I$ is given by $A_1^{f_1}\ldots A_n^{f_n}$ where $f_i = \min(e_i,g_i)$. Hence a prime factorisation is minimal if and only if $0 \leq g_i \leq e_i$ for all $i \in \{1,\ldots ,n\}$.
\end{lemma}

\begin{proof}
Let $I = A_1^{g_1}\ldots A_n^{g_n}$. Then 
\begin{align*}
\Phi(I) &= P_1^{g_1} \ast \ldots  \ast P_n^{g_n} \\
&= P_1^{g_1}\ldots P_n^{g_n} + A \\
&= P_1^{g_1}\ldots P_n^{g_n} + P_1^{e_1}\ldots P_n^{e_n} \\
&= P_1^{\min(g_1, e_1)}\ldots P_n^{\min(g_n,e_n)}
\end{align*}
is the unique prime factorisation of $\Phi(I)$ so $A_1^{\min(g_1, e_1)}\ldots A_n^{\min(g_n,e_n)}$ is the minimal prime factorisation of $I$.
\end{proof}

\begin{corollary}
Let $I = A_1^{i_1}\ldots A_n^{i_n}$ and $J = A_1^{j_1}\ldots A_n^{j_n}$ be minimal prime factorisations of $I,J \in D$. The minimal prime factorisation of $IJ$ is $$A_1^{\min(i_1+j_1, e_1)}\ldots A_n^{\min(i_n+j_n,e_n)}.$$
\end{corollary}

\medskip

We can now apply our methods from Section \ref{semilattice-section}, and in particular Proposition~\ref{D-semilattice-proposition}, to find the semilattice structure of $D$.
\begin{lemma}
Let $I \in D$ with minimal prime factorisation $A_1^{f_1}\ldots A_n^{f_n}$. Then $I \in E(D)$ if and only if $f_i \in \{0,e_i\}$ for all $i \in \{1,\ldots ,n\}$.
\end{lemma}

\begin{proof}
Let $I \in D$ have minimal prime factorisation $A_1^{f_1}\ldots A_n^{f_n}$ so $I^2$ has minimal prime factorisation $A_1^{\min(2f_1, e_1)}\ldots A_n^{\min(2f_n,e_n)}$. If $f_i = 0$ then $\min(2f_i,e_i) = 0 = f_i$ and if $f_i = e_i$ then $\min(2f_i,e_i) = e_i = f_i$ so $I^2 = A_1^{f_1}\ldots A_n^{f_n} = I$.

Conversely, if $I^2 = I$ then $\min(2f_i,e_i) = f_i$ for all $i \in \{1,\ldots ,n\}$. Then if $f_i \leq e_i/2$ we have $2f_i = f_i$ so $f_i = 0$ and if $f_i > e_i/2$ we have $f_i = e_i$. Hence $f_i \in \{0,e_i\}$ for all $i \in \{1,\ldots ,n\}$.
\end{proof}

Let $N = \{1,\ldots,n\}$. The previous lemma shows that an idempotent $e$ is entirely determined by which $A_i$ have a non-zero power $f_i$, and hence we have a bijection $\Lambda: E(D) \to \pn$ given by $\Lambda(e) = \{i \in N | f_i = e_i\}$. If $\pn$ is equipped with the operation of union of sets it then becomes a semilattice and $\Lambda$ can easily seen to be an order isomorphism.  

Note that for every $I \in D$ there exists $e \in E(D)$ such that $I$ has minimal prime factorisation $\prod_{i \in \Lambda(e)}A_i^{f_i}$ where $0 < f_i \leq e_i$ for all $i \in \Lambda(e)$. In fact $I \in D_e$ if and only if its minimal prime factorisation can be written in this way. To see this, it is sufficient to observe that for any prime ideal $A_i$ we have $\varepsilon(A_i) = A_i^{e_i}$ as $\varepsilon$ is a homomorphism.

\begin{proposition}\label{layer-proposition}
Let $D_e$ be a subsemigroup of $D$ for some $e = \prod_{j \in \Lambda(e)} A_j^{e_j} \in E(D)$. Let $I \in D_e$ with minimal prime factorisation $\prod_{j \in \Lambda(e)} A_j^{f_j}$ for $0 < f_j \leq e_j$ and suppose that $I \neq e$. Then for each $i \geq 1$, $I \in {D_e}^i$ if and only if $\min\{f_j | f_j \neq e_j\} \geq i$.
\end{proposition}

Note that $\{f_j | f_j \neq e_j \}$ is non-empty since $I \neq e$. Since $e$ is idempotent, $e \in {D_e}^i$ for all $i \in \n$.

\begin{proof}
By definition, $\prod_{j \in \Lambda(e)} A_j$ divides every element of $D_e$ so $\prod_{j \in \Lambda(e)} A_j^i$ divides every element of ${D_e}^i$. Then for every $I \in {D_e}^i$ there exists a prime factorisation $\prod_{j \in \Lambda(e)} A_j^{g_j}$ with $g_j \geq i$. By Lemma~\ref{minimal-factorisation-lemma} the minimal prime factorisation of $I$, $\prod_{j \in \Lambda(e)} A_j^{f_j}$, has $f_j = \min(e_j, g_j)$ so we have $f_j = g_j \geq i$ for every $f_j \neq e_j$ and hence $\min\{f_j | f_j \neq e_j\} \geq i$.

For the converse, suppose $I$ has minimal prime factorisation $\prod_{j \in \Lambda(e)} A_j^{f_j}$ such that $\min\{f_j | f_j \neq e_j\} \geq i$. Then for each $j \in \Lambda(e)$ either $f_j = e_j$ or $i\le f_j < e_j$. Let $g_j = \max(f_j, i)$ and $J = \prod_{j \in \Lambda(e)} A_j^{g_j}$. Then $J = \left(\prod_{j \in \Lambda(e)} A_j\right)^{i-1}\left(\prod_{j \in \Lambda(e)} A_j^{g_j-(i-1)}\right)$ so, as $g_j-(i-1) > 0$, $J \in {D_e}^i$. If $i < f_j < e_j$ then $g_j = f_j < e_j$. Otherwise, $g_j > f_j = e_j$. In either case, $\min(g_j,e_j) = f_j$ and hence $I$ and $J$ have the same minimal prime factorisation, so $I = J$ and $I \in {D_e}^i$.
\end{proof}

\begin{corollary}
The ideal $I = \prod_{j \in \Lambda(e)} A_j^{f_j}$ with $0<f_j\le e_j$ lies in the $i$-th layer of $D_e$, ${D_e}^i \setminus {D_e}^{i+1}$, if and only if $\min\{f_j | f_j \neq e_j\} = i$.
\end{corollary}

\begin{corollary}
For $e\in E(D)$, $\Base(D_e) = \{e\}$ and the subsemigroup $D_e$ is a stratified semigroup with zero.
\end{corollary}

\medskip

We can now easily prove the main theorem.
\begin{theorem}\label{main-theorem}
Let $S$ be a Dedekind domain, $A\unlhd S$ be an ideal of $S$ and let $R=S/A$. The multiplicative semigroup of $R$ is a semilattice ${\cal S}[E(D);R_e]$ of stratified extensions of groups.
\end{theorem}

\begin{proof}
If $A=\{0\}$ then the result follows from Theorem~\ref{domain-semilattice-theorem}, while if $A=S$ the result is trivial. Henceforth, assume that $\{0\}\ne A\ne S$.

\smallskip

That $R$ is the given semilattice follows immediately from Theorem~\ref{R-semilattice-theorem} and the observation that as every ideal of a quotient of a Dedekind domain is principal, the map $\ed$ is a surjection.

Let $e \in E(D)$ and consider $\Base(R_e)$. By Lemma~\ref{paper1-lemma2}, $\Base(R_e) \subseteq \delta^{-1}(\Base(D_e)) = \delta^{-1}(\{e\})$. Since every ideal of $R$ is principal, there exists some $x \in R$ such that $\delta(x) = (x) = e$ and hence $\Base(R_e) \subseteq \delta^{-1}(\{e\}) = \xux{x}$. By Theorem~\ref{xux-group-theorem}, $\xux{x}$ is a group and hence by Lemma~\ref{paper1-lemma}, $\xux{x} \subseteq \Base(R_e)$ and so $\Base(R_e) = \xux{x}$ and $R_e$ is a stratified extension of a group.
\end{proof}

\medskip

It is clear from Lemma~\ref{paper1-lemma2} that ${R_e}^i \subseteq \delta^{-1}({D_e}^i)$. In fact, we have equality.

\begin{proposition}
Let $e = \prod_{j \in \Lambda(e)} A_j^{e_j} \in E(D)$. Then ${R_e}^i = \delta^{-1}({D_e}^i)$.
\end{proposition}

\begin{proof}
It remains to show that $\delta^{-1}({D_e}^i) \subseteq {R_e}^i$. Let $x \in R_e$ be such that $\delta(x) \in {D_e}^i$. Then $\delta(x)$ has minimal prime factorisation $\prod_{j \in \Lambda(e)} A_j^{f_j}$ where $\min\{f_j | f_j \neq e_j\} = i$. Let $g_j = \max(f_j, i)$. Then $\prod_{j \in \Lambda(e)} A_j^{g_j}$ is a prime factorisation of $\delta(x)$ with $g_j \geq i$ for every $j \in \Lambda(e)$.

Since $S$ is a Dedekind domain, every ideal of $R$ is principal so there exists some $a_j \in R$ such that $\delta(a_j) = (a_j) = A_j$ for every $j \in \Lambda(e)$. Let $y = \prod_{j \in \Lambda(e)} a_j^{g_j}$. Clearly $\delta(y) = \delta(x)$ and so $x {\cal J} y$ and hence $x = yu$ for some $u \in R$. Then
$$
x = \left(\prod_{j \in \Lambda(e)} a_j\right)^{i-1}\left(u \prod_{j \in \Lambda(e)} a_j^{g_j-(i-1)}\right).
$$
Clearly $\ed(\prod_{j \in \Lambda(e)} a_j) = e$, so if $\ed(u \prod_{j \in \Lambda(e)} a_j^{g_j-(i-1)}) = e$ then $x \in {R_e}^i$ as required. Suppose otherwise, so $\ed(u \prod_{j \in \Lambda(e)} a_j^{g_j-(i-1)}) = f$ for some $f \in E(D)$ with $f \neq e$. As each $g_j-(i-1) > 0$, $\ed(\prod_{j \in \Lambda(e)} a_j^{g_j-(i-1)}) = e$ and so $e$ divides $f$ and hence $ef = f$. But then $\ed(x) = e^{i-1}f = f$, a contradiction. Hence $\ed(u \prod_{j \in \Lambda(e)} a_j^{g_j-(i-1)}) = e$ and so $x \in {R_e}^i$ and $\delta^{-1}({D_e}^i) \subseteq {R_e}^i$.
\end{proof}

\begin{corollary}
Let $x \in R_e$. Then $x$ lies in the $i$-th layer of $R_e$ if and only if $\delta(x)$ lies in the $i$-th layer of $D_e$.
\end{corollary}

\begin{corollary}
Let $A \unlhd S$ be a non-zero proper ideal with prime factorisation $P_1^{e_1} \ldots P_n^{e_n}$ and $A_i = P_i/A$. Then 
\begin{enumerate}
\item $R_e$ is a group if and only if $e=\prod_{i\in\Lambda(e)}A_i$.
\item $R$ is a semilattice of groups if and only if $e_1 = \ldots = e_n = 1$.
\item $R$ is a semilattice of groups if and only if $R_{(0)}$ is a group.
\item $E(D)$ is a chain if and only if $n = 1$, in which case it is the two element semilattice.
\end{enumerate}
\end{corollary}
An interesting consequence of this construction is that $S/A$ is a field if and only if $A$ is prime.

\begin{proposition}
Let $R$ be a quotient of a Dedekind domain. Then $R$ is a strong semilattice of semigroups if and only if it is a semilattice of groups.
\end{proposition}

\begin{proof}
It is well known (see, for example, \cite[Theorem 4.2.1]{howie-95}) that a semilattice of groups is a strong semilattice. For the converse, suppose $R$ is a strong semilattice of semigroups. For any $e \in E(D)$ we have $R \geq e$ and so there exists a morphism $\phi_{R,e}: R_R \to R_e$ such that $xy = \phi_{R,e}(x)y$ for any $x \in R_R$ and $y \in R_e$. Since $1 \in R_R$ we have $x = 1x = \phi_{R,e}(1)x$ for every $x \in R_e$. Then $x \in {R_e}^2$ so ${R_e}^2 = R_e$ and hence $R_e$ is a group. As this holds for every $e \in E(D)$, $R$ is a semilattice of groups.
\end{proof}

\medskip

We can summarise the construction of the semilattice of semigroups with this short `recipe'. First we note that we can reduce the amount of calculation required by making use of the following result.

\begin{proposition}\label{xvx-proposition}
Let $S$ be a Dedekind domain, $A$ an ideal of $S$, and $R = S/A$. For all $x \in R$, $xV_x = xU$ where $U$ is the group of units of $R$.
\end{proposition}

\begin{proof}
Note that when $A = S$ the result is trivial and when $A = (0)$ we have $V_x = U$ by cancellativity as $R$ is a domain. We assume henceforth that $A$ is a non-zero proper ideal.

It is readily apparent that $U \subseteq V_x$ for all $x \in R$ and hence $xU \subseteq xV_x$. For the reverse inclusion, we note that it is well known that $R$ is then a principal ideal ring and so, by \cite{kaplansky-49}, if $(a)=(b)$ then $a=bu$ for some $u \in U$. It is easy to see that if $y \in xV_x$ then $(x) = (y)$ and so $y \in xU$ as required.
\end{proof}

Let $S$ be a Dedekind domain and $A \unlhd S$. If $A = S$ then $R=S/A$ is the trivial ring. If $A = (0)$ then $R \cong S$ and so by Theorem~\ref{domain-semilattice-theorem} we have a semilattice of two semigroups. One is $R_R$, the group of units, while the other is $R_{(0)}$, a stratified extension of the trivial group. In the latter case, the elements of layer $i$ are precisely those which can be factorised as a product of $i$ irreducible elements.

\smallskip

Otherwise, let $(0) \neq A \neq S$ be a proper non-zero ideal of $S$ and let $A=P_1^{e_1}\ldots P_n^{e_n}$ be the unique factorisation of $A$ into a product of prime ideals of $S$. Then the semilattice is order isomorphic to ${\cal P}(N)$ and each subset $K\subseteq N$ is associated with an idempotent $e\in E(D)$. The subsemigroup $R_e$ is then a stratified extension of a group where the group is $\Base(R_e) = \delta^{-1}(e)$.

\smallskip

To calculate $R_e$ and $\Base(R_e)$ in a practical setting, we proceed as follows. First, the two easy cases are when $K=\emptyset$, in which case $e=(1+A)$ and $\Base_{R_e}$ is the group of units of $R$, while if $K=N$ then $e=(0+A)$ and the group consists of only the zero of $R$. Suppose now that $\emptyset\subset K\subset N$ and let $A_i = \Phi^{-1}(P_i) = P_i/A$. Then $e = \prod_{i\in K}A_i^{e_i}$ and  since $R$ is a principal ideal ring, if $e=(x+A)$ then $\Base(R_e) = \{xv+A\}$ where $v+A\in V_{x+A}$, and so by Proposition~\ref{xvx-proposition}, $\Base(R_e) = (x+A)R_R$. To determine the stratified structure of $R_e$, note that if $e_i=1$ for all $i\in K$ then $R_e$ is a group and so there are no layers. If at least one of the $e_i>1$ and if for a given subset $K$ and a collection $f_i, i\in K$
$$
\prod_{i\in K}A_i^{f_i} = (y+A)
$$ 
and where $0< f_i\le e_i$ is such that $j = \min\{f_i\;|\;f_i\ne e_i,i\in K\}$, then $\{yv+A\;|\;v+A\in V_{y+A}\}=(y+A)R_R$ is a subset of the $j$-th layer, and moreover the $j$-th layer consists of the union of all such subsets. Note that if $A_i=(a_i+A)$ then $y+A=\prod_{i\in K}(a_i^{f_i}+A)$.

\medskip

\section{Examples}

In this short section we illustrate the above theory by considering a number of examples of Dedekind domains and examining the semilattice and stratified structure of the multiplicative semigroup of both the domain and of a typical quotient of the domain.

\medskip

At the more trivial end of the spectrum, suppose that $S = F$, a field. Then every non-zero element is a unit, so we have $S_{(0)} = \{0\}$ and $S_S = F^\times$. In other words, the multiplicative semigroup of a field is simply a group with zero as expected.

\subsection{The integers}

As a more interesting example, let $S = \z$, the ring of integers. For any $n\in\z$, the sets $\xux{n}$ are (isomorphic to) $\{n,-n\}$ and the units in $\z$ are of course $\pm 1$ and so $S_S$ is the two element group. We know from Theorem~\ref{domain-semilattice-theorem} that $S_{(0)}$ is a stratified extension of the trivial group and the layered structure of $S_{(0)}$ is then easy to establish. The first layer of $S_{(0)}$ consists of every prime integer $p$. The second layer contains all products $pq$ of exactly 2 (not necessarily distinct) primes $p$ and $q$, and in general, layer $n$ consists of all products of exactly $n$ (not necessarily distinct) primes.

\smallskip

Given that $\z$ is a principal ideal domain, then all ideals of $\z$ are of the form $(n)$ for some $n\in\z$. If $R = S/(n)$ then of course $R = \z_n$ the ring of integers modulo $n$. To reduce pedantry we will assume that $\z_n = \{1,\ldots,n\}$. We know from Theorem~\ref{main-theorem} that $R$ is a semilattice of stratified extensions of groups, ${\cal S}[E(D);R_e]$ and that $E(D) \cong {\cal P}(K)$ where $K=\{1,\ldots, k\}$ and where $n=p_1^{e_1}\ldots p_k^{e_k}$ is the prime factorisation of $n$ in $\z$.

First, note that if $(e)\in E(D)$ then we can assume, without loss of generality, that $e = \prod_{i\in I}{p_i^{e_i}}\in\z_n$ where $I = \Lambda((e))\in{\cal P}(K)$.  The base of $R_e$ is $\Base(R_e)=eU_{\overline e}$ and using Theorem~\ref{xux-group-theorem} and Lemma~\ref{cancellation-lemma}, we deduce that $\z_n/({\overline e})\cong \z_{n/e}$ and that
$$
eU_{\overline e}\cong U_{n/e}
$$
where $U_{n/e}$ is the group of units in $\z_{n/e}$. If $(e)$ is square-free then $R_e=U_{n/e}$ otherwise $R_e$ is a stratified extension of $U_{n/e}$ with height $m=\max\{e_j|j\in\Lambda(e)\}-1$. In this case, the structure of the individual layers of $R_e$ is more complicated to describe is general, but essentially if $x$ is in the $i$-th layer of $R_e$, $1\le i\le m$, then
$$
x = \prod_{j\in\Lambda(e)}p_j^{g_j}u
$$
where $u\in U_n$ and $0<g_j\le e_j$ and $\min\{g_j|g_j\ne e_j\} = i$. Note that $\{g_j|g_j\ne e_j\}\ne\emptyset$ as otherwise $x\in\Base(R_e)$.

\smallskip

As an example, if $n=12 = 2^2\times3$, then
$$
E(D) = \{(12),(4),(3),(1)\}
$$
and we have four subsemigroups

$R_{(12)} = \{6,12\}$ where $\Base(R_{(12)}) = \{12\}$. 

$R_{(4)} = \{2,4,8,10\}$ where $\Base(R_{(4)}) = \{4,8\}$ and $\{2,10\}$ forms layer 1.

$R_{(3)} = \{3,9\}$ which is a group.

$R_{(1)} = \{1,5,7,11\}$ which is the group of units mod 12.

The semilattice structure can be pictured as
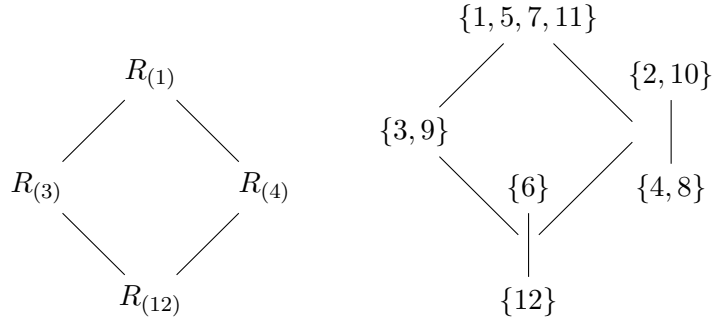
\begin{center}
\begin{tikzpicture}[scale=0.75]
\node (max) at (0,4) {$R_{(1)}$};
  \node (a) at (-2,2) {$R_{(3)}$};
  \node (c) at (2,2) {$R_{(4)}$};
  \node (min) at (0,0) {$R_{(12)}$};
  \draw (min) -- (c) -- (max) -- (a) -- (min);
\end{tikzpicture}
\hskip2em
\begin{tikzpicture}[scale=0.75]
\node (max) at (0,4) {$\{1,5,7,11\}$};
  \node (a) at (-2,2) {$\{3,9\}$};
  \node (c) at (2.5,3) {$\{2,10\}$};
  \node (d) at (2.5,1) {$\{4,8\}$};
  \node (e) at (2,2) {};
  \node (min) at (0,0) {};
  \node (f) at (0,1) {$\{6\}$};
  \node (g) at (0,-1) {$\{12\}$};
  \draw (min) -- (e) -- (max) -- (a) -- (min);
  \draw (d) -- (c);
  \draw (f) -- (g);
\end{tikzpicture}
\end{center}
Notice that the semilattice will always be a finite Boolean algebra and the stratification structure is wholly dependent on the prime power factorisation of $n$. 

\subsection{The $p-$adic integers}

Let $S$ be the $p$-adic integers. There are a number of ways to view $p-$adic numbers but we consider $S$ to consist of formal sums
$$
S = \left\{\sum_{i\ge 0}a_ip^i\;|\;0\le a_i\le p-1\right\}
$$
with arithmetic performed in the usual formal manner. For more detail we refer the reader to~\cite{gouvea-20}. The expression $\sum_{i\ge 0}a_ip^i$ is also known as the {\it $p-$adic expansion} of the relevant number. 

It is easy to demonstrate that the units in $S$ are the elements where $a_0\ne0$ in the $p-$adic expansion and that non-unit elements have the form $p^k u$ where $u$ is a unit of $S$ and $k \in \n$.  It is well-known that $S$ forms a principal ideal domain and so from Theorem~\ref{domain-semilattice-theorem} we deduce that $S$ is a (2-element) semilattice of stratified extensions of groups. $S_{(1)}$ is the group of units and $\Base(S_{(0)}) = \{0\}$. It follows from the definition of $S$ that the proper non-zero ideals are those of the form $(p^k)$ for $k \in \n$, and so clearly $D_{(0)}$ is isomorphic to the infinite monogenic semigroup with zero. Since $S$ is a principal ideal domain, it follows from Proposition~\ref{pid-proposition} that ${S_{(0)}}^i = \delta^{-1}({D_{(0)}}^i)$ for all $i \in \n$ and so the $i$-th layer of $S_{(0)}$ consists of exactly the elements of the form $p^i u$ where $u$ is a unit.

\smallskip

Every non-zero proper ideal $A \unlhd S$ has the form $(p^k) = (p)^k$ for some $k \in \n$. This means that $S/A$ is isomorphic to the ring of integers modulo $p^k$. Clearly $(p)$ is a prime ideal and so $R = S/A$ is a 2-element semilattice of stratified extensions of groups, consisting of the group of units $R_{(1+A)}$ and the semigroup $R_{(0+A)}$. The latter is a stratified semigroup with zero and $k-1$ non-zero layers. For each $1\le i\le k-1$ the $i$-th layer consists of elements of the form $p^iu+A$ where $u$ is a unit of $S$.

\subsection{Rings of algebraic integers}
We now consider rings consisting of algebraic integers and as a specific example we shall consider the ring $S = \z[\sqrt{-5}]$. It is well know that rings of this nature are Dedekind domains but are not always principal ideal domains. In fact, $2+\sqrt{-5}$ is an example of an element which can easily be shown to be irreducible but not prime. 
If $A$ is an ideal of $\z[\sqrt{-5}]$ define a `norm' on $S/A$ by $N(z+A) = (z+A)({\overline z}+A) = z{\overline z}+A$, where $\overline z$ is the conjugate of $z$. It is easy to check that $N$ is multiplicative and that $z+A$ is a unit in $S/A$ if $N(z+A) = \pm1+A$.

From section~\ref{section-4}, $S$ is a 2-element semilattice of the group of units, $S_{(1)}=\{1,-1\}$, and a stratified semigroup with 0, $S_{(0)}$. Since
$$
2+\sqrt{-5} = 9\times(-2) + (-1+4\sqrt{-5})\times(-\sqrt{-5}),
$$
$$
(-1+4\sqrt{-5}) = (2+\sqrt{-5})^2\text{ and }9=(2-\sqrt{-5})(2+\sqrt{-5}),
$$
it follows that $(3,2+\sqrt{-5})^2 =(9,-1+4\sqrt{-5})= (2+\sqrt{-5})$ and so although $2+\sqrt{-5}$ is in the first layer of $S_{(0)}$ (being irreducible), $(2+\sqrt{-5})$ is not in the first layer of $D_{(0)}$. The layer structure of $S_{(0)}$ is not so easy to determine, as clearly $a+b\sqrt{-5}$ is in the $i$-th layer of $S_{(0)}$ if and only if it can be written as a product of $i$ irreducible elements.

However determining the structure of a quotient of $S$ is slightly easier, as we need only factorise a single ideal of $S$ into a product of prime ideals. As an illustrative example, let us consider
$$
A = (10, 5+5\sqrt{-5}) = (2, 1+\sqrt{-5})(5,\sqrt{-5})^2
$$
and let $R=S/A$. It is easy to show that $(2, 1+\sqrt{-5})$ and $(5,\sqrt{-5})$ are both prime ideals of $\z[\sqrt{-5}]$. In fact
$$
(2, 1+\sqrt{-5}) = \{a+b\sqrt{-5}\;|\;a\equiv b\text{ mod }2\}\text{ and }(5,\sqrt{-5}) = \{5a+b\sqrt{-5}\;|\; a,b \in \z\},
$$
while
$$
A = (10,5+5\sqrt{-5}) = \{5a+5b\sqrt{-5}\;|\;a\equiv b\text{ mod }2\}.
$$
It is easy to check that the ring $R$ has cardinality 50. In what follows, we shall frequently simplify the notation by working modulo $A$ and write the element $a+b\sqrt{-5}+A$ of $R$ as simply $a+b\sqrt{-5}$. We shall also assume a particular set of residues by taking $0\le a\le 9$ and $0\le b\le 4$.

Note from the comments preceding Proposition~\ref{layer-proposition} that $|E(D)|=|{\cal P}(\{1,2\})|$ and it can then be easily verified that 
$(5,\sqrt{-5})^2/A = (5)$ and that $(2,1+\sqrt{-5})/A = (6)$ and so
$$
E(D) = \{(0), (1), (5), (6)\}.
$$

It follows that $R_{(1)}$ and $R_{(6)}$ are groups and $R_{(0)}$ and $R_{(5)}$ are stratified extensions of groups, each with a height of $1$. For each $e \in E(D)$, the group $\Base(R_e)$ is equal to $\xux{x}$ where $(x) = e$ and hence isomorphic to $U_{\overline{x}}$.

In practical terms, $R_{(1)}=R_R$ is the group of units of $R$, and using norms we can deduce that $|R_R|=20$ and in fact
$$
R_R = \{a+b\sqrt{-5}\;|\;a\not\equiv b\text{ mod }2, a\not\equiv0\text{ mod }5\}.
$$

For $R_{(5)}$, it follows that $\Base(R_{(5)}) = \{5v\;|\;v\in R_R\} = 5R_R = \{5\}$. To find the elements in layer 1 of $R_{(5)}$ we note that $(5,\sqrt{-5})/A = (\sqrt{-5})$ and so layer 1 is
$$
\delta^{-1}((\sqrt{-5})) = \{(\sqrt{-5})v\;|\;v\in  R_R\} = \{\sqrt{-5},3\sqrt{-5},7\sqrt{-5},9\sqrt{-5}\}.
$$

For $R_{(6)}$, it follows that 
$$
\Base(R_{(6)}) = 6R_R = \{a+b\sqrt{-5}\;|\;a\equiv b\text{ mod }2, a\not\equiv0\text{ mod }5\}.
$$
Note that $|\Base(R_{(6)})|=20$ also.

Finally, the layer 1 in $R_{(0)}$ can be calculated in the same way as for $R_{(5)}$ using the fact that $(2, 1+\sqrt{-5})/A\;(5,\sqrt{-5})/A = (5+\sqrt{-5})$. It then follows easily that the first layer of $R_{(0)}$ is
$$
(5+\sqrt{-5})R_R = \{2\sqrt{-5},4\sqrt{-5},5+\sqrt{-5},5+3\sqrt{-5}\}.
$$

\subsection{Integers revisited}

For a final example we return to a less complicated ring in order to demonstrate a more complicated layer structure. Let $S = \z$, $A = (6000) = (2)^4(3)(5)^3$ and $R = S/A$. Working modulo $A$, let $e \in E(D)$ be the ideal $(2000) = (2)^4(5)^3$. Then $R_e$ is a stratified extension of a group with height 3. Applying our previous results, we see that $\Base(R_e) = \delta^{-1}((2000)) = \{2000u\;|\; u \in R_R\} = 2000R_R$.

\medskip

For the layers, note that layer 1 of $D_e$ consists of $(2)(5)$, $(2)^2(5)$, $(2)^3(5)$, $(2)^4(5)$, $(2)(5)^2$ and $(2)(5)^3$. Layer 1 of $R_e$ is hence the union of $10R_R$, $20R_R$, $40R_R$, $80R_R$, $50R_R$ and $250R_R$.

\smallskip

Proceeding in a similar fashion, layer 2 of $R_e$ is the union of $100R_R$, $200R_R$, $400R_R$ and $500R_R$, while layer 3 is simply the set $1000R_R$.

\end{document}